\documentclass[11pt]{amsart}
\usepackage{amsmath, amssymb, latexsym}
\usepackage{mathrsfs}
\usepackage[all, knot]{xy}
\xyoption{arc}

\theoremstyle{definition}

\numberwithin{equation}{section}

\numberwithin{equation}{section}

\newtheorem{theorem}{Theorem}[section]
\newtheorem{corollary}[theorem]{Corollary}
\newtheorem{lemma}[theorem]{Lemma}
\newtheorem{proposition}[theorem]{Proposition}

\theoremstyle{definition}

\newcommand{\HOM}{\text{HOM}}

\newcommand{\B}{\mathbb{B}}

\newcommand{\Q}{\mathbb{Q}}

\newcommand{\Z}{\mathbb{Z}}
\newcommand{\N}{\mathbb{N}}

\newcommand{\K}{\mathbb{K}}
\newcommand{\Ss}{\mathcal{S}}
\newcommand{\ii}{\textit{\textbf{i}}}
\newcommand{\jj}{\textit{\textbf{j}}}
\newcommand{\kk}{\textit{\textbf{k}}}
\newcommand{\Seq}{\text{Seq}}
\newcommand{\Seqd}{\text{Seqd}}
\newcommand{\Pol}{\mathscr{P}}

\newcommand{\Ind}{\text{Ind}}
\newcommand{\Res}{\text{Res}}
\newcommand{\Span}{\text{Span}}
\newcommand{\gdim}{\mathbf{Dim}}
\newcommand{\Mod}{\text{Mod}}
\newcommand{\pMod}{\text{pMod}}
\newcommand{\fMod}{\text{fMod}}
\newcommand{\Ch}{\text{Ch}}
\newcommand{\Sh}{\text{Sh}}
\newcommand{\hd}{\text{hd}}
\newcommand{\soc}{\text{soc}}
\newcommand{\Max}{\text{max}}

\newcommand{\genO}[3] % generator 1
{
\fontsize{9}{9}\selectfont
\xy
(0,5)*{}; (0,-5)*{} **\dir{-};
(4,0)*{\cdots};
(8,5)*{}; (8,-5)*{} **\dir{-};
(12,0)*{\cdots};
(16,5)*{}; (16,-5)*{} **\dir{-};
(0,-7)*{#1}; (8,-7)*{#2}; (16,-7)*{#3};
\endxy
\fontsize{10}{10}\selectfont}
\newcommand{\genX}[3] % generator x
{
\fontsize{9}{9}\selectfont
\xy
(0,5)*{}; (0,-5)*{} **\dir{-};
(4,0)*{\cdots};
(8,5)*{}; (8,-5)*{} **\dir{-};
(8,0)*{\bullet}; (12,0)*{\cdots};
(16,5)*{}; (16,-5)*{} **\dir{-};
(0,-7)*{#1}; (8,-7)*{#2}; (16,-7)*{#3};
\endxy
\fontsize{10}{10}\selectfont}
\newcommand{\genT}[4] % generator tau
{\fontsize{9}{9}\selectfont
\xy
(0,5)*{}; (0,-5)*{} **\dir{-};
(4,0)*{\cdots};
(5,5)*{}; (12,-5)*{} **\dir{-};
(12,5)*{}; (5,-5)*{} **\dir{-};
(13,0)*{\cdots};
(17,5)*{}; (17,-5)*{} **\dir{-};
(0,-7)*{#1}; (5,-7)*{#2}; (12,-7)*{#3}; (17,-7)*{#4};
\endxy
\fontsize{10}{10}\selectfont}

\newcommand{\dcross}[2]
{\fontsize{9}{9}\selectfont
\xy
(0,7)*{}="T1"; (7,7)*{}="T2";
(0,-7)*{}="B1"; (7,-7)*{}="B2";
"T1"; "B1" **\crv{(11, 0)};
"T2"; "B2" **\crv{(-4,0)};
(0,-9)*{#1}; (7,-9)*{#2};
\endxy
\fontsize{10}{10}\selectfont}
\newcommand{\dcrossA}[2]
{\fontsize{9}{9}\selectfont
\xy
(0,5)*{}; (0,-5)*{} **\dir{-};
(6,5)*{}; (6,-5)*{} **\dir{-};
(0,-7)*{#1}; (6,-7)*{#2};
\endxy
\fontsize{10}{10}\selectfont}
\newcommand{\dcrossL}[3]
{ \fontsize{9}{9}\selectfont
\xy
(0,5)*{}; (0,-5)*{} **\dir{-};
(6,5)*{}; (6,-5)*{} **\dir{-};
(0,0)*{\bullet}; (-5,0)*{#1}; (0,-7)*{#2}; (6,-7)*{#3};
\endxy
\fontsize{10}{10}\selectfont}
\newcommand{\dcrossR}[3]
{\fontsize{9}{9}\selectfont
\xy
(0,5)*{}; (0,-5)*{} **\dir{-};
(6,5)*{}; (6,-5)*{} **\dir{-};
(6,0)*{\bullet}; (11,0)*{#1}; (0,-7)*{#2}; (6,-7)*{#3};
\endxy
\fontsize{10}{10}\selectfont}

\newcommand{\LU}[3]
{\fontsize{9}{9}\selectfont
\xy
(0,5)*{}; (8,-5)*{} **\dir{-} \POS?(.25)="x";
(8,5)*{}; (0,-5)*{} **\dir{-};
"x"*{\bullet}; "x"+(-2,2)*{#1}; (0,-7)*{#2}; (8,-7)*{#3}; \endxy
\fontsize{10}{10}\selectfont}
\newcommand{\RD}[3]
{\fontsize{9}{9}\selectfont
\xy
(0,5)*{}; (8,-5)*{} **\dir{-} \POS?(.75)="x";
(8,5)*{}; (0,-5)*{} **\dir{-};
"x"*{\bullet}; "x"+(2,2)*{#1}; (0,-7)*{#2}; (8,-7)*{#3}; \endxy
\fontsize{10}{10}\selectfont}
\newcommand{\RU}[3]
{\fontsize{7}{7}\selectfont
\xy
(0,5)*{}; (8,-5)*{} **\dir{-};
(8,5)*{}; (0,-5)*{} **\dir{-} \POS?(.25)="x";
"x"*{\bullet}; "x"+(2,2)*{#1}; (0,-7)*{#2}; (8,-7)*{#3}; \endxy
\fontsize{10}{10}\selectfont}
\newcommand{\LD}[3]
{\fontsize{9}{9}\selectfont
\xy
(0,5)*{}; (8,-5)*{} **\dir{-};
(8,5)*{}; (0,-5)*{} **\dir{-} \POS?(.75)="x";
"x"*{\bullet}; "x"+(-2,2)*{#1}; (0,-7)*{#2}; (8,-7)*{#3}; \endxy
\fontsize{10}{10}\selectfont}

\newcommand{\BraidL}[3]
{\fontsize{9}{9}\selectfont
\xy
(0,7)*{}="T1"; (6,7)*{}="T2";  (12,7)*{}="T3";
(0,-7)*{}="B1"; (6,-7)*{}="B2"; (12,-7)*{}="B3";
"T1"; "B3" **\crv{(0,0)&(12,0)}; "T3"; "B1" **\crv{(12,0)&(0,0)};
"T2"; "B2" **\crv{(6,6)&(3,4.5)&(0,1.5)&(0,-1.5)&(3,-4.5)&(6,-6)};
(0,-9)*{#1}; (6,-9)*{#2};(12,-9)*{#3};
\endxy
\fontsize{10}{10}\selectfont}
\newcommand{\BraidR}[3]
{\fontsize{9}{9}\selectfont
\xy
(0,7)*{}="T1"; (6,7)*{}="T2";  (12,7)*{}="T3";
(0,-7)*{}="B1"; (6,-7)*{}="B2"; (12,-7)*{}="B3";
"T1"; "B3" **\crv{(0,0)&(12,0)}; "T3"; "B1" **\crv{(12,0)&(0,0)};
"T2"; "B2" **\crv{(6,6)&(8,4.5)&(12,1.5)&(12,-1.5)&(8,-4.5)&(6,-6)};
(0,-9)*{#1}; (6,-9)*{#2};(12,-9)*{#3};
\endxy
\fontsize{10}{10}\selectfont}
\newcommand{\threeDotStrands}[3]
{ \fontsize{9}{9}\selectfont
\xy
(0,7)*{}; (0,-7)*{} **\dir{-};
(5,7)*{}; (5,-7)*{} **\dir{-};
(10,7)*{}; (10,-7)*{} **\dir{-};
(0,0)*{\bullet};(10,0)*{\bullet}; (-2,0)*{c};(20,0)*{-a_{ij}-1-c}; (0,-9)*{#1}; (5,-9)*{#2};(10,-9)*{#3};
\endxy
\fontsize{10}{10}\selectfont}

\newcommand{\Shuffle}
{\fontsize{9}{9}\selectfont
\xy
(-1,15)*{}; (11,15)*{} **\dir{-};
(14,15)*{}; (31,15)*{} **\dir{-};
(-1,0)*{}; (31,0)*{} **\dir{-};
(0,15)*{}; (0,0)*{} **\dir{-};
(5,15)*{};(10,0)*{} **\crv{(5,7.5)&(10,7.5)};
(10,15)*{}; (25,0)*{} **\crv{(10,7.5)& (25,7.5)};
(15,15)*{}; (5,0)*{} **\crv{(15,7.5)&(5,7.5)};
(20,15)*{}; (15,0)*{} **\crv{(20,7.5)&(15,7.5)};
(25,15)*{}; (20,0)*{} **\crv{(25,7.5)&(20,7.5)};
(30,15)*{}; (30,0)*{} **\dir{-};
(-5,7.5)*{u};(-2,15)*{\ii}; (32,15)*{\jj};(-2,0)*{\kk};
\endxy
\fontsize{11}{10}\selectfont}

\newcommand{\dcrossI}[2]
{\fontsize{9}{9}\selectfont
\xy
(0,6)*{}="T1"; (6,6)*{}="T2";
(0,-6)*{}="B1"; (6,-6)*{}="B2";
"T1"; "B1" **\crv{(10, 0)};
"T2"; "B2" **\crv{(-4,0)};
(0,-8)*{#1}; (6,-8)*{#2};
\endxy
\fontsize{10}{10}\selectfont}
\newcommand{\BraidLI}[3]
{\fontsize{9}{9}\selectfont
\xy
(0,6)*{}="T1"; (5,6)*{}="T2";  (10,6)*{}="T3";
(0,-6)*{}="B1"; (5,-6)*{}="B2"; (10,-6)*{}="B3";
"T1"; "B3" **\crv{(0,0)&(10,0)}; "T3"; "B1" **\crv{(10,0)&(0,0)};
"T2"; "B2" **\crv{(5,5)&(3,4)&(0,1.5)&(0,-1.5)&(3,-4)&(5,-5)};
(0,-8)*{#1}; (5,-8)*{#2};(10,-8)*{#3};
\endxy
\fontsize{10}{10}\selectfont}
\newcommand{\BraidRI}[3]
{\fontsize{9}{9}\selectfont
\xy
(0,6)*{}="T1"; (5,6)*{}="T2";  (10,6)*{}="T3";
(0,-6)*{}="B1"; (5,-6)*{}="B2"; (10,-6)*{}="B3";
"T1"; "B3" **\crv{(0,0)&(10,0)}; "T3"; "B1" **\crv{(10,0)&(0,0)};
"T2"; "B2" **\crv{(5,5)&(7,4)&(10,1.5)&(10,-1.5)&(7,-4)&(5,-5)};
(0,-8)*{#1}; (5,-8)*{#2};(10,-8)*{#3};
\endxy
\fontsize{10}{10}\selectfont}

\title[]
{Quiver Hecke algebras for Borcherds-Cartan datum}

\author[]{Bolun Tong}
\address{Harbin Engineering University,
Harbin, China}
\email{tbl\_2019@hrbeu.edu.cn}

\author[]{Wan Wu}
\address{Chengdu University of Information Technology, Chengdu, China}
\email{wuwan1818@163.com}



\begin{document}
\maketitle
\begin{abstract}
We introduce a family of quiver Hecke algebras which give a categorification of quantum Borcherds algebra associated to an arbitrary Borcherds-Cartan datum.
\end{abstract}

\section*{\textbf{Introduction}}

\vskip 2mm
Quiver Hecke algebras, also known as Khovanov-Lauda-Rouquier algebras, were discovered independently by Khovanov-Lauda \cite{KL2009,KL2011} and Rouquier \cite{Rou}, and their representation theory  is shown to be closely related to quantum groups. In Kac-Moody type,  the category of finitely generated graded projective modules over quiver Hecke algebras give a categrification of corresponding quantum groups. Varagnolo-Vasserot \cite{VV2011} and Rouquier \cite{Rou1} proved that, under this connection, the indecomposable projective modules  correspond to the Lusztig's canonical basis, and their irreducible modules correspond to the dual canonical basis.

In this paper, we apply Khovanov-Lauda's  categorification theory   to the  quantum Borcherds algebras, which were introduced by Kang in \cite{Kang}. Given a  Borcherds-Cartan datum consisting of an index set $I$  and a symmetrizable Borcherds-Cartan matrix $A=(a_{ij})_{i,j\in I}$, we construct a family of graded algebras $R(\nu)$ ($\nu\in \N[I]$) associated to it, using the  braid-like planar diagrams, and give a faithful polynomial representation for each $R(\nu)$. When $i$ is a real index in $I^+$, the degenerated algebras $R(ni)$ for $n\in \N$ are exactly the nil-Hecke algebras $NH_n$ as usual. When $i$ is an imaginary index in $I^-$, $R(ni)$ is generated by `dots' $x_1,\dots,x_n$  and `intersections' $\tau_1,\dots,\tau_{n-1}$, with local relations expressed diagrammatically:

 $$\dcrossI{}{}=0\quad\quad \BraidLI{}{}{}  =  \BraidRI{}{}{}\quad\quad \LU{{}}{}{}    =  \RD{{}}{}{}\quad\quad \LD{{}}{}{}     =    \RU{{}}{}{} $$
 We show that $R(ni)$ has a unique graded irreducible module in this case, which is a one-dimensional trivial module denoted by $V(i^n)$. The induction of  two irreducible modules, $\Ind V(i^n)\otimes V(i^m)$, has an irreducible head isomorphic to $V(i^{n+m})$.

  We then form the Grothendieck group $K_0(R)=\bigoplus_{\nu\in\N[I]}K_0(R(\nu))$ of the category of finitely generated graded projective modules. Let $U^-$ be the negative part of the quantum Borcherds algebra associated to the given Borcherds-Cartan datum. A classical framework given in \cite{KL2009,KL2011} leads to an injective homomorphism $\Gamma_{\Q(q)}:U^-\rightarrow\Q(q)\otimes_{\Z[q,q^{-1}]}K_0(R)$. The surjectivity of $\Gamma_{\Q(q)}$ follows from the arguments in \cite[Chapter 5]{K2005} and  \cite[Section 3.2]{KL2009}. But we need to modify some proofs there since the $R((n+m)i)$-module $\Ind  V(i^n)\otimes V(i^m)$ is not irreducible  again when $i\in I^-$. Finally, the map $\Gamma_{\Q(q)}$ induces a $\Z[q,q^{-1}]$-algebra isomorphism $\Gamma:_\mathcal AU^-\rightarrow K_0(R)$, where $_\mathcal AU^-$ is the $\mathcal A$-form of $U^-$.

In \cite{KOP2012}, Kang, Oh and Park gave a categorification of this algebra with the condition  $a_{ii}\neq 0$  in the Borcherds-Cartan matrix $A$. Our construction of the quiver Hecke algebras is different from their and applies to categorifying the  quantum group with an  arbitrary Borcherds-Cartan datum.

\vskip 6mm
\section{\textbf{Preliminaries}}

\subsection{$\Z$-gradings}\

\vskip 2mm

We fix an algebraically closed field $\K$. Let $A$ be a $\Z$-graded algebra over $\K$. For a graded $A$-module $M=\bigoplus_{n\in \Z}M_n$, its graded dimension is defined by
$$\gdim M:=\sum_{n\in\Z}(\text{dim}M_n)q^n,$$
where $q$ is a formal variable. For $m\in \Z$, we denote by $M\{m\}$ the graded $A$-module obtained from $M$ by putting $(M\{m\})_n=M_{n-m}$.  For $f(q)=\sum_{n\in \Z}a_nq^n\in\Z[q,q^{-1}]$, define $M^{f}:=\bigoplus_{n\in\Z}(M\{n\})^{\oplus a_n}$, we have $\gdim M^{f}=f(q)\cdot\gdim M$.
 \vskip 2mm
Given two graded $A$-modules $M$ and $N$, we denote by $\text{Hom}_A(M,N)$  the ${\K}$-vector space of grading-preserving homomorphisms and define the $\Z$-graded vector space $\text{HOM}_A(M,N)$ by
$$\text{HOM}_A(M,N)=\bigoplus_{n\in\Z}\text{Hom}_A(M\{n\},N)=\bigoplus_{n\in\Z}\text{Hom}_A(M,N\{-n\}).$$

\vskip 3mm

\subsection{Negative parts of quantum Borcherds algebras}\

\vskip 2mm
Let $I$ be a finite index set. A Borcherds-Cartan datum $(I,A,\cdot)$ consists of
\begin{itemize}
\item [(a)] an integer-valued matrix $A=(a_{ij})_{i,j\in I}$ satisfying
\begin{itemize}
\item [(i)] $a_{ii}=2,0,-2,-4,\dots$,
\item [(ii)] $a_{ij}\in \Z_{\leq 0}$ for $i\neq j$,
\item [(iii)] there is a diagonal matrix $D=\text{diag}(r_i\in\Z_{>0}\mid i\in I)$ such that $DA$ is symmetric.
\end{itemize}
\item [(b)] a symmetric bilinear form $\nu,\nu'\mapsto \nu\cdot \nu'$ on $\Z[I]$ taking values in $\Z$, such that $$i\cdot j=r_ia_{ij}=r_ja_{ji} \ \ \text{for all} \ i,j\in I.$$
\end{itemize}
 For such a datum, we assign a graph $\Lambda$ with vertices set $I$ and an edge between $i$ and $j$ if $i\cdot j\neq 0$.
\vskip 2mm

We set $I^+=\{i\in I\mid a_{ii}=2\}$ and $I^-=\{i\in I\mid a_{ii}\leq 0\}$. Let $q$ be an indeterminate. For each $i\in I$, let $q_i=q^{r_i}$. For $i\in I^+$ and $n\in \N$, we define
 $$[n]_i=\frac{q_i^n-q_i^{-n}}{q_i-q_i^{-1}}\ \text{and} \ [n]_i!=[n]_i[n-1]_i\cdots [1]_i.$$

The negative part $U^-$ of the quantum Borcherds algebra associated to a Borcherds-Cartan datum $(I,A,\cdot)$ is an associative algebra over $\Q(q)$ with generators $f_i$ $(i\in I)$ and the defining relations
\begin{equation*}
\begin{aligned}
& \sum_{r+s=1-i\cdot j}(-1)^rf_i^{(r)}f_jf_i^{(s)}=0 \quad \text{for} \ i\in I^+,j\in I\ \text{and}\ i\neq j,\\
& f_if_j=f_jf_i \quad \text{for}\ i\in I,j\in I\ \text{and}\ i\cdot j=0.
\end{aligned}
\end{equation*}
Here we denote $f_i^{(n)}=f_i^n/[n]_i!$ for $i\in I^+$ and $n\in \N$. The algebra $U^-$ is $\N[I]$-graded by assigning $\text{deg}(f_i)=i$.

\vskip 2mm
Define a twisted multiplication on $U^-\otimes U^-$ by
$$(x_1\otimes x_2)(y_1\otimes y_2)=q^{-|x_2|\cdot |y_1|}x_1y_1\otimes x_2y_2,$$
for homogeneous $x_1,x_2,y_1,y_2$. By \cite[Proposition 2.4]{SV2001}, we have an algebra homomorphism $\rho:U^-\rightarrow U^-\otimes U^-$ given by $\rho{(f_i)}=f_i\otimes 1+1\otimes f_i$ ($i\in I$) with respect to the above algebra structure on $U^-\otimes U^-$, and a nondegenerate symmetric bilinear form $\{ \ , \ \}: U^-\times U^-\rightarrow \Q(q)$ satisfying the following propoties
\begin{itemize}
\item[(i)] $\{x, y\} =0$ if $|x| \neq |y|$,
\item[(ii)] $\{1,1\} = 1$,
\item[(iii)] $\{f_{i}, f_{i}\} = (1-q_i^2)^{-1}$ for all $i\in I$,
\item[(iv)] $\{x, yz\} = \{\rho(x), y \otimes z\}$  for $x,y,z
\in U^-$.
\end{itemize}

\vskip 2mm

Let $\mathcal A=\Z[q,q^{-1}]$ be the ring of Laurent polynomials. The $\mathcal A$-form $_{\mathcal A} U^-$ is the $\mathcal A$-subalgebra of $U^-$ generated by the divided powers $f_i^{(n)}$ for $i\in I^+,n\in \Z_{\geq 0}$ and $f_i$ for $i\in I^-$.

\vskip 6mm
\section{\textbf{Algebras $R(\nu)$ for Borcherds-Catan datum}}

\vskip 2mm

As in \cite{KL2009}, we construct ${\K}$-algebras $R(\nu)$ ($\nu\in \N[I]$) for Borcherds-Cartan datum by braid-like planar diagrams, in which each strand is labelled by an element of $I$ and can carry dots. These diagrams are invariant when planar isotropy is considered.

\subsection{Definition and polynomial representation}\
\vskip 2mm

Given a Borcherds-Cartan datum $(I,A,\cdot)$. We fix a $\nu=\sum_{i\in I}\nu_ii\in \N[I]$ with $\text{ht}{(\nu)}=n$. Let $\text{Seq}(\nu)$ be the set of all sequences $\ii=i_1i_2\dots i_n$ in $I$ such that $\nu=i_1+i_2\cdots +i_n$. We define the homogeneous generators of $R(\nu)$ by diagrams:
\begin{equation*}
\begin{aligned}
&\ \ 1_{\ii}=\genO{i_1}{i_k}{i_n} \quad \text{for} \ \ii=i_1i_2\dots i_n\in \text{Seq}(\nu) \ \text{with} \ \text{deg}(1_{\ii})=0,\\
&  x_{k,\ii}=\genX{i_1}{i_k}{i_n} \quad \text{for} \ \ii\in \text{Seq}(\nu),1\leq k\leq n \ \text{with} \ \text{deg}(x_{k,\ii})=2r_{i_k},\\
& \tau_{k,\ii}=\genT{i_1}{i_k}{i_{k+1}}{i_n} \quad \text{for} \ \ii\in \text{Seq}(\nu), 1\leq k\leq n-1 \ \text{with} \ \text{deg}(\tau_{k,\ii})=-i_k\cdot i_{k+1}.
\end{aligned}
\end{equation*}
The multiplication $A\cdot B$ of two diagrams $A,B$ is given by concatenation if the bottom sequence of $A$ coincides with the top sequence of $B$, and otherwise is zero.
The local relations of $R(\nu)$ are defined as follows:

\begin{align}
 \dcross{i}{j} \  =  \  \begin{cases} \quad \quad \quad \quad \quad \  0 & \text{ if } i= j, \\
                    \\
                    \quad  \quad \quad \quad \ \ \dcrossA{i}{j} & \text{ if } i\ne j\ \text{and} \ i\cdot j =0, \\
                    \\
                    \   \dcrossL{-a_{ij}}{i}{j} \ + \ \dcrossR{-a_{ji}}{i}{j}  & \text{ if } i\ne j \ \text{and} \ i\cdot j \ne 0,
                  \end{cases}
\end{align}

\begin{equation}
\begin{aligned}
 \LU{{}}{i}{i} \   -  \ \RD{{}}{i}{i}\ =\ \dcrossA{i}{i} \quad\quad\quad \LD{{}}{i}{i}   \ - \ \RU{{}}{i}{i}\ =\ \dcrossA{i}{i}  \quad \text{ if } i \in I^+,
\end{aligned}
\end{equation}

\begin{equation}
\begin{aligned}
 \LU{{}}{i}{i} \    =  \ \RD{{}}{i}{i}  \quad\quad\quad\LD{{}}{i}{i}   \  =   \ \RU{{}}{i}{i}  \quad \text{ if } i \in I^-,
\end{aligned}
\end{equation}

\begin{equation}
\begin{aligned}
 \LU{{}}{i}{j} \   =  \ \RD{{}}{i}{j}  \quad\quad\quad \LD{{}}{i}{j}   \ =  \ \RU{{}}{i}{j}  \quad \text{ if } i \ne j,
\end{aligned}
\end{equation}

\begin{equation}\label{S1}
\quad \quad\quad\quad \quad\BraidL{i}{j}{i} \  -  \ \BraidR{i}{j}{i} \ =  {\sum_{c=0}^{-a_{ij}-1}} \ \threeDotStrands{i}{j}{i} \ \ \text{ if } i\in I^+, i\ne j \ \text{and} \ i\cdot j \ne 0,
 \end{equation}
 \begin{equation}
\BraidL{i}{j}{k} \  = \ \BraidR{i}{j}{k} \quad\quad \text{otherwise}.
\end{equation}

\vskip 2mm

For $\ii,\jj\in \Seq(\nu)$, we set $_{\jj}R(\nu)_{\ii}=1_{\jj}R(\nu)1_{\ii}$, $P_{\ii}=R(\nu)1_{\ii}$ and $_{\jj}P=1_{\jj}R(\nu)$. We have $R(\nu)=\bigoplus_{\ii,\jj} {_{\jj}R(\nu)_{\ii}}$ and $P_{\ii}$ (resp. $_{\jj}P$) is a gr-projective left (resp. right) $R(\nu)$-module.

\vskip 2mm

Choose an orientation for each edge of the graph $\Lambda$. For $\ii\in \Seq(\nu)$, set $$\mathscr{P}_{\ii}={\K}[x_1(\ii),\dots,x_n(\ii),y_1(\ii),\dots,y_n(\ii)]$$ and form the $\K$-vector space
$\mathscr{P}_{\nu}=\bigoplus_{\ii\in \Seq(\nu)}\mathscr{P}_{\ii}$.
Let $S_n=\left<s_1,\dots,s_{n-1}\right>$ be the symmetric group. For each $\omega\in S_n$, define the operators
\begin{equation*}
\begin{aligned}
& \omega:x_a(\ii)\mapsto x_{\omega (a)}(\omega(\ii)),\ y_a(\ii)\mapsto y_{\omega(a)}(\omega(\ii)),\\
& \widetilde\omega:x_a(\ii)\mapsto x_{\omega(a)}(\omega(\ii)),\ y_a(\ii)\mapsto y_a(\omega(\ii)).\\
\end{aligned}
\end{equation*}

We then define an action of $R({\nu})$ on $\Pol_{\nu}$ as follows. $_{\jj}R(\nu)_{\ii}$ acts by $0$ on $\Pol_{\kk}$ if $\ii\neq\kk$. For $f\in \Pol_{\ii}$, $1_{\ii}\cdot f=f$, $x_{k,\ii}\cdot f=x_k{(\ii)}f$ and
\begin{equation}\label{rep}
\tau_{k,\ii}\cdot f=\begin{cases} {s_k}f & \text{if}\ i_k\neq i_{k+1}\ \text{and}\ i_k\cdot i_{k+1}=0, \\
\frac{f-\widetilde s_kf}{x_k{(\ii)}-x_{k+1}(\ii)} & \text{if}\ i_k= i_{k+1}\in I^+,\\
\frac{\widetilde s_kf-s_kf}{y_k{(\ii)}-y_{k+1}(\ii)} & \text{if}\ i_k= i_{k+1}\in I^-,\\
s_kf & \text{if}\ i_k\leftarrow i_{k+1},\\
 (x_k(s_k\ii)^{-a_{ji}}+x_{k+1}(s_k\ii)^{-a_{ij}})s_kf & \text{if}\ i_k\rightarrow i_{k+1}.
\end{cases}
\end{equation}

\vskip 2mm

\begin{proposition}\label{P}{\it
$\Pol_\nu$ is a $R(\nu)$-module with the action defined above.}
\begin{proof}
This can be  obatined immediately by checking the  relations of $R(\nu)$.
\end{proof}
\end{proposition}

\vskip 3mm
\subsection{Algebras $R(ni)$ for $i\in I^-$}\

\vskip 2mm
In this section, we consider the graph $\Lambda$ with one vertex $i$ and the corresponding algebras $R{(ni)}$ for $n\in\N$. If $i\in I^+$,  $R{(ni)}$ is isomorphic to the nil-Hecke algebra $NH_n$, its algebraic structure and graded representations are well-known (cf. \cite[Example 2.2]{KL2009}). So we consider $i\in I^-$ only. In this case, $R(ni)$ is isomorphic to the ${\K}$-algebra $R^i_n$  with generators $x_1,\dots,x_n$ of degree $2r_i$ and $\tau_1,\dots,\tau_{n-1}$ of degree $-i\cdot i$, subject to the following relations:
\begin{equation*}
\begin{aligned}
& x_kx_t=x_tx_k\  \text{for all} \ 1\leq k,t\leq n,\\
& \tau_k^2=0,\ \tau_k\tau_{k+1}\tau_{k}=\tau_{k+1}\tau_{k}\tau_{k+1},\ \tau_k\tau_t=\tau_t\tau_k \ \text{if}\ |k-t|>1,\\
& x_k\tau_k=\tau_kx_{k+1},\ \tau_kx_k=x_{k+1}\tau_k,\\
& \tau_kx_t=x_t\tau_k  \ \text{if}\ t\neq k,k+1.
\end{aligned}
\end{equation*}
We simply write $R_n$ for $R^i_n$ if there is no ambiguity. Since $\tau_k$ ($1\leq k\leq n-1$) satisfy the braid relations, for each $\omega\in S_n$, we can define $\tau_\omega=\tau_{k_1}\cdots\tau_{k_r}$ if $\omega$ has a reduced expression $\omega=s_{k_1}\cdots s_{k_r}$.
\vskip 2mm

Let $\Pol_n=\K[x_1,\dots,x_n,y_1\dots,y_n]$, and let $\partial_k:\Pol_n\rightarrow \Pol_n$ ($1\leq k\leq n-1$) be the linear operators given by
$$f\mapsto\frac{\widetilde{s}_kf-s_kf}{y_k-y_{k+1}}.$$
Here $\widetilde{s}_k$ acts on $f$ by interchanging $x_k$ and $x_{k+1}$, $s_k$ acts on $f$ by interchanging $x_k$ and $x_{k+1}$ and interchanging $y_k$ and $y_{k+1}$ simultaneously. According to Proposition \ref{P},  $\Pol_n$  is a left $R_n$-module with the action of $x_k$ by multiplication and the action of $\tau_\omega$ by $\partial_\omega$.%, where $\partial_k:\Pol_n\rightarrow \Pol_n$ is the Demazure operator defined by
%$$\partial_k(f)=\frac{f-s_kf}{y_k-y_{k+1}}.$$

\vskip 2mm

\begin{proposition}\label{S2}{\it
The algebra $R_n$ has a basis $\{x_1^{r_1}\cdots x_n^{r_n}\tau_\omega\mid \omega\in S_n,r_1,\dots,r_n\geq 0\}$.}
\begin{proof}
We show that these elements act on $\Pol_n$ independently. Suppose that we have a non-trivial linear combination
$\sum_{\omega;r_1,\dots,r_n}k_{\omega;r_1,\dots,r_n}x_1^{r_1}\cdots x_n^{r_n}\tau_\omega$ acts by zero on $\Pol_n$. Choose a minimal length element $\omega$ such that $k_{\omega;r_1,\dots,r_n}\neq 0$ for some $r_1,\dots,r_n$. Let $\omega_0=s_1(s_1s_2)\cdots(s_1\cdots s_{n-1})$ be the longest element in $S_n$ and write $\omega_0=\omega\omega'$. By applying this linear combination to $\partial_{\omega'}(y_1^{n-1}y_2^{n-2}\cdots y_{n-1})$ , we get
$$ \textstyle{\sum_{r_1,\dots,r_n}}k_{\omega;r_1,\dots,r_n}x_1^{r_1}\cdots x_n^{r_n}=0,$$
which implies $k_{\omega;r_1,\dots,r_n}=0$ for all $r_1,\dots,r_n\geq 0$, a contradiction.
\end{proof}
\end{proposition}
\vskip 2mm

Since there is an anti-automorphism of $R_n$ taking $x_k$ to $x_k$ and $\tau_k$ to $\tau_k$, we see that $\{\tau_\omega x_1^{r_1}\cdots x_n^{r_n}\mid \omega\in S_n,r_1,\dots,r_n\geq 0\}$ is also a basis of $R_n$. We identify the polynomial algebra $P_n=\K[x_1,\dots,x_n]$ with the subalgebra of $R_n$ generated by $x_1,\dots,x_n$. Let $P_n^{S_n}$ be the subalgebra consisting of all symmetric polynomials in $P_n$.

\vskip 2mm

\begin{proposition}\label{S3}
The center of $R_n$ is $P_n^{S_n}$.
\begin{proof}
The proof is an analogue of Theorem 3.3.1 in \cite{K2005}. Let $z=\sum_{\omega\in S_n}f_\omega\tau_\omega$ be a center element. Assume that $\omega\neq 1$ with $f_\omega\neq 0$, then there exists $k\in \{1,\dots,n\}$ such that $\omega (k)\neq k$. But this implies $x_kz-zx_k=\sum_{\omega\in S_n}f_\omega(x_k-x_{\omega(k)})\tau_\omega\neq 0$, a contradiction. Thus $z\in P_n$. Write $z=\sum_{i,j\geq 0}p_{ij}x_1^ix_2^j$ with $p_{ij}\in \K[x_3,\dots,x_n]$. Now $\tau_1z=z\tau_1$ implies $p_{ij}=p_{ji}$ for each $i,j$. Hence $z$ is symmetric in $x_1$ and $x_2$. Similarly, we can show that $z$ is symmetric in $x_k$ and $x_{k+1}$ for all $1\leq k\leq n-1$.
\end{proof}
\end{proposition}

\vskip 2mm

We denote by $L$ the one-dimensional trivial module over $P_n$. Note that $L$ is the unique gr-irreducible $P_n$-module, up to a degree shift. Let $$\overline L=R_n\otimes_{P_n}L=\bigoplus_{\omega\in S_n}\tau_\omega\otimes L,$$ which is a graded left $R_n$-module. Since $x_k\tau_\omega=\tau_\omega x_{\omega^{-1}(k)}$ for any $k$ and $\omega$, we have
$$ x_1 \cdot \overline{L}=x_2\cdot \overline{L}=\cdots=x_n\cdot \overline{L}=0.$$

\vskip 2mm

Fix a nonzero $v\in L$, then $\{ \tau_\omega\otimes v\mid \omega\in S_n\}$ is a basis of $\overline L$.
\vskip 2mm
\begin{lemma}{\it
$\overline{L}$ has a unique (graded) irreducible submodule $V=\Span \{\tau_{\omega_0}\otimes v\}$ with the action of $R_n$ trivially.}
\begin{proof}
Let $W$ be a nonzero submodule of $\overline L$. Assume $m=\sum_{\omega}k_\omega\tau_\omega\otimes v$ is a nonzero element of $W$. Choose a minimal length element $\omega$ such that $k_{\omega}\neq 0$ and write $\omega_0=\omega'\omega$, we have $\tau_{\omega'}m=k_\omega\tau_{\omega_0}\otimes v\in W$. This shows that each nonzero submodule contains $V$. Moreover, $V$ itself is a (graded) $R_n$-module.
\end{proof}
\end{lemma}

\vskip 2mm
\begin{lemma}\label{M}{\it
$\overline{L}$ has a unique (graded) maximal submodule $M=\Span \{\tau_{\omega}\otimes v\mid \omega\neq 1\}$. In particular, $\overline{L}/M\simeq V$ as  $R_n$-modules.}
\begin{proof}
It's obvious that $M$ is a (graded) maximal submodule of $\overline{L}$. For any nonzero submodule $W$ of $\overline L$, if $W$ contains an element $z$ of the form
$$z=1\otimes v+\sum_{\omega\in S_n,\omega\neq 1}k_{\omega}\tau_\omega\otimes v,$$
then we choose a minimal length element $\omega\neq 1$ with $k_{\omega}\neq 0$  and obtain $z_1=z-k_\omega\tau_\omega z\in W$. Note $z_1$ is of the form
$$z_1=1\otimes v+\sum_{l(\omega')\geq l(\omega), \omega'\neq \omega}c_{\omega'}\tau_{\omega'}\otimes v$$
for some $c_{\omega'}\in \K$.
By repeating this process, one can deduce that $1\otimes v\in W$. Therefore, if a submodule $W\neq \overline{L}$, then $W\subseteq M$. The lemma is proved.
\end{proof}
\end{lemma}
\vskip 2mm
\begin{theorem}{\it
$V$ is the unique gr-irreducible module over $R_n$, up to a degree shift.}
\begin{proof}
Let $N$ be a gr-irreducible $R_n$-module, then $N$ contains a $P_n$-submodule isomorphic to $L\{r\}$ for some $r\in \Z$. Since we have the graded isomorphism
$$\text{HOM}_{R_n}(\overline{L},N)\simeq\text{HOM}_{P_n}(L,\text{HOM}_{R_n}(R_n,N))\simeq\text{HOM}_{P_n}(L,N)\neq 0,$$
there exist a surjective graded homomorphism $\overline{L}\twoheadrightarrow N\{-r\}$ by the irreducibility of $N$. By Lemma \ref{M}, we have $N\{-r\}\simeq \overline{L}/M\simeq V$.
\end{proof}
\end{theorem}

\vskip 2mm

We shall denote by $V(i^n)$ the unique gr-irreducible $R(ni)$-module for $i\in I$, which is a one-dimensional trivial module for $i\in I^-$ by arguments above. Recall that, for $i\in I^+$, $V(i^n)$ is isomorphic to $\Ind L= NH_n\otimes_{P_n}L$, up to a degree shift.

\vskip 6mm

\section{\textbf{Basic properties and representation theory of $R(\nu)$}}

\vskip 2mm

This section follows \cite{KL2009} and \cite{KL2011} closely. We list our main results without proof as they can be proved step by step according to \cite{KL2009} and \cite{KL2011} with  appropriate deformations.

\subsection{Basis and center}\
\vskip 2mm

For $\ii,\jj\in \Seq(\nu)$, set $_\jj\Ss_\ii=\{\omega\in S_n\mid \omega(\ii)=\jj\}$. We fix a reduced expression  for each $\omega\in _\jj\Ss_\ii$, which determines a unique element $\widehat\omega_\ii\in _{\jj}R(\nu)_{\ii}$, and set
$$_{\jj}\mathcal B_{\ii}=\{\widehat\omega_\ii\cdot x_{1,\ii}^{r_1}\cdots x_{n,\ii}^{r_n}\mid \omega\in _\jj\Ss_\ii , r_1,\dots,r_n\in \N\}.$$

 \vskip 2mm
\begin{proposition}\label{basis}
{\it $_{\jj}\mathcal B_{\ii}$ is a basis of $_{\jj}R(\nu)_{\ii}$. Moreover, $\Pol_\nu$ is a faithful $R(\nu)$-module with the actions given in (\ref{rep}).}
\begin{proof}
This proposition follows from Proposition \ref{S2} and the standard arguments in \cite[Theorem 2.5]{KL2009}.
\end{proof}
\end{proposition}

\vskip 2mm
 Assume $\nu=\nu_1i_1+\cdots+\nu_ti_t$ such that $i_1,\dots,i_t$ are all distinct and $\nu_k>0$. By Proposition \ref{S3} and Theorem 2.9 in \cite{KL2009}, we describe the center $Z(R{(\nu)})$ of $R{(\nu)}$ as follows.
\vskip 2mm

\begin{proposition}\label{Cen}
{\it $Z(R(\nu))\simeq\bigotimes^t_{k=1}\K[z_1,\dots,z_{\nu_k}]^{S_{\nu_k}}$.
}\end{proposition}

\vskip 2mm
$R(\nu)$ is a free $Z(R(\nu))$-module of rank $(n!)^2$. It is also a graded free $Z(R(\nu))$-module of finite rank. We have
 $$\gdim Z(R(\nu))=\prod ^t_{k=1}\left(\prod^{\nu_k}_{c=1}\frac{1}{1-q_{i_k}^{2c}}\right),$$
 and $\gdim R(\nu)\in \Z[q,q^{-1}]\cdot \gdim Z(R(\nu))$.
\vskip 3mm

\subsection{Grothendieck groups and bilinear pairings}\
\vskip 2mm

Let $R(\nu)$-$\Mod$ be the category of finitely generated graded $R(\nu)$-modules, and let $R(\nu)$-$\fMod$ (resp. $R(\nu)$-$\pMod$) be the full subcategory of $R(\nu)$-$\Mod$ of finite-dimensional (resp. finitely generated projective) $R(\nu)$-modules.

Since $R(\nu)$ is Laurentian by Proposition \ref{basis}, there are only finitely many gr-irreducible $R(\nu)$-module, up to isomorphism and degree shifts. All gr-irreducible $R(\nu)$-module are finite-dimensional. Moreover, if $S$ is a gr-irreducible $R(\nu)$-module, then $\underline{S}$ is a irreducible $R(\nu)$-module by forgetting the grading.

Let $\B_{\nu}$ be the set of equivalence classes (under isomorphism and degree shifts) of
gr-irreducible $R(\nu)$-modules. The Grothendieck group $G_0(R(\nu))$ of $R(\nu)$-$\fMod$ is a free $\Z[q,q^{-1}]$-module with the basis $[S_b]_{b\in\B_{\nu}}$, where $q[M]=[M\{1\}]$ for $[M]\in G_0(R(\nu))$.  Each $S_b$ has a unique gr-indecomposable projective cover $P_b$. The Grothendieck group $K_0(R(\nu))$ of $R(\nu)$-$\pMod$ is a free $\Z[q,q^{-1}]$-module with the basis $[P_b]_{b\in\B_{\nu}}$.

\vskip 3mm

Let $\psi:R(\nu)\rightarrow R(\nu)$ be the anti-involution of $R(\nu)$ by flipping the diagrams about horizontal axis. For $P\in R(\nu)$-$\pMod$, let $\overline{P}=\HOM(P,R(\nu))^\psi$ be the gr-projective left $R(\nu)$-module with the action twisted by $\psi$. This gives a self-equivalence of $R(\nu)$-$\pMod$ and induces a $\Z[q,q^{-1}]$-antilinear involution of $K_0(R(\nu))$ denoted again by $^-$.

Define the $\Z[q,q^{-1}]$-bilinear pairing
$( \ , \ ):K_0(R(\nu))\times G_0(R(\nu))\rightarrow \Z[q,q^{-1}]$ by
$$([P],[M])=\gdim (P^\psi\otimes_{R(\nu)}M)=\gdim \HOM_{R(\nu)}(\overline{P},M).$$
Since $\K$ is algebraically closed, $G_0(R(\nu))$ and $K_0(R(\nu))$ are dual $\Z[q,q^{-1}]$-module under this pairing. There is also a symmetric $\Z[q,q^{-1}]$-bilinear form $( \ , \ ):K_0(R(\nu))\times K_0(R(\nu))\rightarrow \Z ((q))$ defined in the same way.

\vskip 3mm

\subsection{Character and quantum Serre relations}\

\vskip 2mm
For $M\in R(\nu)$-$\Mod$, define the character of $M$ as $$\Ch M=\sum_{\ii\in \Seq(\nu)}\gdim (1_{\ii}M)\ii.$$

We denote by $\Seqd(\nu)$ the set of sequences $\ii$ of $\nu$ with the `divided powers' for $i\in I^+$. Such a sequence is of the form
$$\ii=j_1\dots j_{p_0}i_1^{(n_1)}k_1\dots k_{p_1}i_2^{(n_2)}\dots i_t^{(n_t)}l_1\dots l_{p_t}, $$
where $i_1,\dots,i_t\in I^+$ and $\ii$ is of weight $\nu$.

For $i\in I^+$ and $n> 0$, let $e_{i,n}$ be the primitive idempotent of $R(ni)$ corresponding to  the element $x_1^{n-1}x_2^{n-2}\cdots x_{n-1}\partial _{\omega_0}$ of $NH_n$. For each $\ii\in \Seqd(\nu)$,
we assign the following idempotent of $R(\nu)$
$$1_{\ii}=1_{j_0\dots j_{p_0}}\otimes e_{i_1,n_1}\otimes 1_{k_1\dots k_{p_1}}\otimes e_{i_2,n_2}\otimes \cdots \otimes e_{i_t,n_t}\otimes 1_{l_1\dots l_{p_t}}.$$

\vskip 2mm
We abbreviate $\ii=\dots i_1^{(n_1)}\dots i_2^{(n_2)} \dots i_t^{(n_t)}\dots $ and $1_{\ii}=\cdots \otimes e_{i_1,n_1}\otimes \cdots \otimes e_{i_2,n_2}\otimes \cdots \otimes e_{i_t,n_t}\otimes \cdots$ for simplicity, and denote
$$\widehat{\ii}=\dots \underbrace{i_1\dots i_1}_{n_1}\dots \underbrace{i_2\dots i_2}_{n_2}\dots \underbrace{i_t\dots i_t}_{n_t}\dots\in \Seq(\nu) .$$
Let $\ii != [n_1]_{i_1}!\cdots [n_t]_{i_t}!$ and $\left<\ii \right>=\sum_{k=1}^t\frac{n_k(n_k-1)}{2}r_{i_k}$, we have by the structure of nil-Hecke algebra
$$\gdim (1_{\widehat{\ii}}M)=q^{-\left<\ii \right>}\ii !\cdot\gdim ({1_\ii M}).$$

\vskip 2mm

For $\ii\in \Seqd(\nu)$, let $_\ii P=1_\ii R(\nu)\{-\left<\ii \right>\}$ and $P_\ii = R(\nu)\psi(1_\ii)\{-\left<\ii \right>\}$. We have the following proposition which gives a categorification  of quantum Serre relations in $U^-$.

\begin{proposition}\label{Serre}
{\it For $i\in I^+$, $j\in I$ and $i\neq j$. Let $m=1-a_{ij}$. We have an isomorphism of graded left $R(\nu)$-modules
$$\bigoplus^{\lfloor \frac{m}{2} \rfloor}_{c=0}P_{i^{(2c)}ji^{(m-2c)}}\simeq\bigoplus^{\lfloor \frac{m-1}{2} \rfloor}_{c=0}P_{i^{(2c+1)}ji^{(m-2c-1)}}.$$
Moreover, for $i,j\in I$ and $i\cdot j=0$, we have an isomorphism $P_{ij}\simeq P_{ji}$.
}
\begin{proof}
The proof is the same as the `Box' calculations in \cite{KL2011}.
\end{proof}
\end{proposition}

\vskip 3mm

\subsection{Induction and Restriction}\
\vskip 2mm

As in \cite[Section 2.6]{KL2009}, we define the induction and restriction functors as
\begin{equation*}
\begin{aligned}
& \Ind^{\nu+\nu'}_{\nu,\nu'}:R(\nu)\otimes R(\nu')\text{-}\Mod\rightarrow R(\nu+\nu')\text{-}\Mod,\ M\mapsto R(\nu+\nu')1_{\nu,\nu'}\otimes_{R(\nu)\otimes R(\nu')}M,\\
&\Res^{\nu+\nu'}_{\nu,\nu'}: R(\nu+\nu')\text{-}\Mod\rightarrow R(\nu)\otimes R(\nu')\text{-}\Mod,\ M\mapsto 1_{\nu,\nu'}M,
\end{aligned}
\end{equation*}
where $1_{\nu,\nu'}=1_\nu\otimes 1_{\nu'}$. Since $R(\nu+\nu')1_{\nu,\nu'}$ is a free graded right $R(\nu)\otimes R(\nu')$-module, the functors $\Ind^{\nu+\nu'}_{\nu,\nu'}$ and $\Res^{\nu+\nu'}_{\nu,\nu'}$ are both exact and take projective modules to projective modules. For $\ii\in \Seqd(\nu)$ and $\jj\in\Seqd(\nu')$, we have by the definition $\Ind_{\nu,\nu'}P_\ii\otimes P_\jj\simeq P_{\ii\jj}$.

\vskip 2mm

For $\ii\in \Seq(\nu)$, $\jj\in\Seq(\nu')$ and $\kk\in \Seq(\nu+\nu')$, we denote by  $\Sh(\ii,\jj;\kk)$ the set of all shuffles $u\in _{\ii\jj}R(\nu+\nu')_\kk$ from $\ii,\jj$ to $\kk$

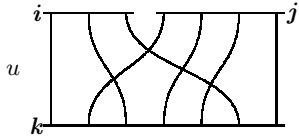
$$\Shuffle$$
\vskip 2mm

The gr-projective $R(\nu)\otimes R(\nu')$-module $\Res_{\nu,\nu'}P_\kk$ has the following decomposition
$$\Res_{\nu,\nu'}P_\kk\simeq\bigoplus_{\ii\in\nu,\jj\in\nu', u\in \Sh(\ii,\jj;\kk)}P_\ii\otimes P_{\jj}\{|u|\}.$$
For $M\in R(\nu)\text{-}\Mod, N\in R(\nu')\text{-}\Mod$ and $\kk\in \Seq(\nu+\nu')$, we have the following equality, so called the `Quantum Shuffle Lemma'
$$\gdim (1_\kk \Ind_{\nu,\nu'}M\otimes N)=\sum_{\ii\in\nu,\jj\in\nu', u\in \Sh(\ii,\jj;\kk)}q^{|u|}\gdim (1_\ii M)\cdot \gdim (1_\jj N).$$

\vskip 2mm

\begin{proposition}\label{Mackey} (``Mackey Theorem") {\it
Let $\nu,\nu',\mu,\mu'\in \N[I]$ with $\nu+\nu'=\mu+\mu'$. For $M\in R(\mu)\text{-}\Mod, N\in R(\mu')\text{-}\Mod$, we have a filtration of  $\Res_{\nu,\nu'}\Ind_{\mu,\mu'}M\otimes N$ with subquotients over all $\lambda\in \N[I]$ such that $\nu-\lambda, \mu'-\lambda, \nu'+\lambda-\mu' \in\N[I]$, which are isomorphic to
$$\Ind^{\nu,\nu'}_{\nu-\lambda,\lambda,\nu'+\lambda-\mu',\mu'-\lambda} {}^\diamond(\Res^{\mu,\mu'}_{\nu-\lambda,\nu'+\lambda-\mu',\lambda,\mu'-\lambda}M\otimes N )\{-\lambda\cdot(\nu'+\lambda-\mu')\}.$$
Here if $\Res M\otimes N=Q_1\otimes Q_2\otimes Q_3\otimes Q_4$, then ${}^\diamond(\Res M\otimes N)=Q_1\otimes Q_3\otimes Q_2\otimes Q_4$.
}
\begin{proof}
The proof is the same as \cite[Proposition 2.8]{KL2009}.
\end{proof}
\end{proposition}

\vskip 3mm

\subsection{Bialgebra $K_0(R)$} \
\vskip 2mm
Let $R=\bigoplus_{\nu\in\N[I]}R(\nu)$ and form the following categories of $R$-modules
\begin{equation*}
 R\text{-}\fMod=\bigoplus_{\nu\in \N[I]}R(\nu)\text{-}\fMod, \ R\text{-}\pMod=\bigoplus_{\nu\in \N[I]}R(\nu)\text{-}\pMod.
\end{equation*}
The Grothendieck groups of $R$-$\fMod$ (resp. $R$-$\pMod$) is given by $G_0(R)=\bigoplus_{\nu\in \N[I]}G_0(R(\nu))$ (resp. $K_0(R)=\bigoplus_{\nu\in \N[I]}K_0(R(\nu))$). By summing up all $\nu,\nu'$, the induction and restriction functors induce the following $\Z[q,q^{-1}]$-linear maps
$$\widetilde{\Ind}:K_0(R)\otimes K_0(R)\rightarrow K_0(R), \ \widetilde{\Res}:K_0(R)\rightarrow K_0(R)\otimes K_0(R).$$
Now, $K_0(R)$ becomes a $\Z[q,q^{-1}]$-algebra with the multiplication given by  $xy:=\widetilde{\Ind}(x\otimes y)$ for all $x,y\in K_0(R)$. If we equip $K_0(R)\otimes K_0(R)$ with a twisted algebra structure via
$$(x_1\otimes x_2)(y_1\otimes y_2)=q^{-|x_2|\cdot |y_1|}x_1y_1\otimes x_2y_2,$$
then $\widetilde{\Res}$ is a $\Z[q,q^{-1}]$-algebra homomorphism by Mackey's Theorem given in Proposition \ref{Mackey}.

\vskip 2mm

Extend the bilinear pairings in Section 3.2 to $K_0(R)\times K_0(R)$ and to $K_0(R)\times G_0(R)$ by requiring $([M],[N])=0$ if $M\in R(\nu)\text{-}\pMod$, $N\in R(\mu)\text{-}\pMod$ (or $R(\mu)\text{-}\fMod$) with $\nu\neq \mu$. We have the following proposition from the definition.

\begin{proposition}\label{Bi}
{\it The symmetric bilinear form on $K_0(R)$ satisfies
\begin{itemize}
\item[(i)] $(1,1) = 1$,
\item[(ii)] $([P_i], [P_j]) = \delta_{ij}(1-q_i^2)^{-1}$ for all $i,j\in I$,
\item[(iii)] $(x, yz) = (\widetilde{\Res}(x), y \otimes z)$  for $x,y,z
\in K_0(R),$
\end{itemize}
where $1=\K$ as a module over $R(0)=\K$.
}\end{proposition}

\vskip 6mm

\section{\textbf{Categorication of $U^-$ and $_{\mathcal A}U^-$} }
\vskip 2mm

As in \cite[Proposition 3.4]{KL2009}, we connect the Grothendieck group $K_0(R)$ with the half part of quantum Borcherds algebra $U^-$ as follows. Let $K_0(R)_{\Q(q)}=\Q(q)\otimes_{\Z[q,q^{-1}]}K_0(R)$. By Proposition \ref{Serre}, we have a well-defined $\Q(q)$-algebra homomorphism
$$\Gamma_{\Q(q)}:U^-\rightarrow K_0(R)_{\Q(q)}$$
given by $\Gamma_{\Q(q)}(f_i)=[P_i]$ for all $i\in I$. By Proposition \ref{Bi}, the bilinear form $\{ \ , \ \}$ on $U^-$ and the form $( \ , \ )$ on $K_0(R)_{\Q(q)}$ take same values under $\Gamma_{\Q(q)}$, that is
$$(\Gamma_{\Q(q)}(x),\Gamma_{\Q(q)}(y))=\{x,y\} \ \ \text{for} \ x,y\in U^-.$$
Thus $\Gamma_{\Q(q)}$ is injective by the non-degeneracy of $\{ \ , \ \}$. Moreover, it induces an injective $\Z[q,q^{-1}]$-algebra homomorphism
$\Gamma:_{\mathcal A}U^-\rightarrow K_0(R)$.

\vskip 2mm

In the rest of this section, we shall prove the surjectivity of $\Gamma_{\Q(q)}$ and $\Gamma$ using the frameworks given in \cite[Chapter 5]{K2005} and in \cite[Section 3.2]{KL2009}. Recall that the one-dimensional trivial module $V(i^n)$ is the unique gr-irreducible $R(ni)$-module for $i\in I^-$.

\vskip 2mm

\begin{lemma}
{\it Let $i\in I^-$ and let $(m_1,\dots,m_r)$ be a composition of $n$.
 \begin{itemize}
\item[(i)] $\Res^n_{m_1,\dots,m_r}V(i^n)\simeq V(i^{m_1})\otimes \cdots \otimes V(i^{m_r})$,
\item[(ii)] $\Res^n_{n-1}V(i^n)=V(i^{n-1})$,
\item[(iii)] $\Ind^{n+m}_{n,m}V(i^n)\otimes V(i^m)$ has a unique (graded) maximal submodule. In particular, the graded head $\hd\Ind^{n+m}_{n,m}V(i^n)\otimes V(i^m)$ is irreducible.
\end{itemize}
}
\begin{proof}
The assertions (i) and (ii) are obvious. We shall prove (iii). Let $D_{(n,m)}$ be the set of minimal length left $S_n\times S_m$-coset representatives in $S_{n+m}$, then $\Ind^{n+m}_{n,m}V(i^n)\otimes V(i^m)$ has a basis $\{\tau_\omega\otimes v\mid \omega\in D_{n,m}\}$ for a nonzero $v\in V(i^n)\otimes V(i^m)$. The $\K$-vector space $\Span\{\tau_\omega\otimes v\mid \omega\in D_{n,m},\omega\neq 1\}$ is a maximal submodule of $\Ind^{n+m}_{n,m}V(i^n)\otimes V(i^m)$ since for $\lambda\in S_{n+m}$ and $\omega\in D_{n,m}$ ($\omega\neq 1$), if $\lambda\omega\in S_n\times S_m$, we must have $l(\lambda\omega)<l(\lambda)+l(\omega)$. The uniqueness follows from the same argument in Lemma \ref{M}.
\end{proof}
\end{lemma}

\vskip 2mm
For $i\in I$ and $n\geq 0$, define the functor $$\Delta_{i^n}:R(\nu)\text{-}\Mod\rightarrow R(\nu-ni)\otimes R(ni)\text{-}\Mod,\ M\mapsto (1_{\nu-ni}\otimes 1_{ni})M.$$
By Frobenius reciprocity, we have for $M\in R(\nu)$-$\Mod$ and $N\in R(\nu-ni)$-$\Mod$
\begin{equation}\label{Frob}
\HOM_{R(\nu)}(\Ind_{\nu-ni,ni}N\otimes V(i^n),M)\simeq \HOM_{R(\nu-ni)\otimes R(ni)}(N\otimes V(i^n),\Delta_{i^n}M).\end{equation}
\vskip 2mm
Let $\varepsilon_i(M)=\Max\{n\geq 0\mid \Delta_{i^n}M\neq 0\}$ be the number of the largest $i$-tail in sequence $\kk$ such that $1_\kk M\neq 0$.

\vskip 2mm

\begin{lemma}\label{L1}
{\it Let $i\in I$ and  $M\in R(\nu)$-$\fMod$ be a gr-irreducible module. If $N\otimes V(i^n)$ is a gr-irreducible submodule of $\Delta_{i^n}M$ for some $0\leq n\leq \varepsilon_{i}(M)$, then $\varepsilon_i(N)=\varepsilon_i(M)-n$.
}
\begin{proof}
Let $\varepsilon_i(N)=a$ and $\varepsilon_i(M)=b$, there exists a sequence $\jj i^{a}\in \Seq({\nu-ni})$ such that $1_{\jj i^{a}}N\neq 0$. Hence $1_{\jj i^{a}}\otimes 1_{i^n}(N\otimes V(i^n))=1_{\jj i^{a+n}}(N\otimes V(i^n))\neq 0$. It follows that $b\geq a+n$.

By Frobenius reciprocity (\ref{Frob}) and the irreducibility of $M$, $M$ is a quotient of $\Ind_{\nu-ni,ni}N\otimes V(i^n)$. The exactness of $\Delta_{i^n}$ implies $\Delta_{i^b}M$ is a quotient of $\Delta_{i^b}\Ind_{\nu-ni,ni}N\otimes V(i^n)$, we get $\Delta_{i^b}\Ind_{\nu-ni,ni}N\otimes V(i^n)\neq 0$. On the other hand, we have $\varepsilon_i(\Ind_{\nu-ni,ni}N\otimes V(i^n))=a+n$ by the Shuffle Lemma. Therefore $b\leq a+n$.
\end{proof}
\end{lemma}
\vskip 2mm
\begin{lemma}\label{L2}
{\it Let $i\in I$ and $N\in R(\nu)$-$\fMod$ be a gr-irreducible module with $\varepsilon_i(N)=0$. Set $M=\Ind_{\nu,ni}N\otimes V(i^n)$. Then
\begin{itemize}
\item[(i)] $\Delta_{i^n}M\simeq N\otimes V(i^n)$,
\item[(ii)] $\hd M$ is gr-irreducible with $\varepsilon_i(\hd M)=n$,
\item[(iii)] all other composition factors $L$ of $M$ have $\varepsilon_i(L)<n$.
\end{itemize}
}
\begin{proof}
In the case of $i\in I^+$, the lemma has been proved in \cite[Lemma 3.7]{KL2009}. We now consider the cases $i\in I^-$.

(i) By Frobenius reciprocity (\ref{Frob}) and the irreducibility of $N\otimes V(i^n)$, we have $N\otimes V(i^n)\hookrightarrow \Delta_{i^n} M$ as a graded submodule. Assume $\text{ht}(\nu)=m$, then $\Ch V(i^n)=i^n$ and $\Ch N=\sum_{\jj\in \nu,j_m\neq i}\gdim(1_\jj N)\jj$. By Shuffle Lemma, we have
$$\Ch M=\sum_{\kk\in \nu+ni}\left(\sum_{\jj\in \nu,j_m\neq i, u\in\Sh(\jj,i^n;\kk)}q^{|u|}\gdim(1_{\jj}N)\right)\kk.$$
It follows that $$\Ch (\Delta_{i^n}M)=\sum_{\jj\in \nu,j_m\neq i}\gdim(1_\jj N)\jj i^n=\Ch (N\otimes V(i^n)).$$
Hence $\Delta_{i^n}M\simeq N\otimes V(i^n)$.

(ii) For any nonzero quotient $Q$ of $M$, we have $N\otimes V(i^n)\hookrightarrow \Delta_{i^n} Q$ by Frobenius reciprocity (\ref{Frob}). Assume we have the decomposition
 $$\hd M=M/{J^{\text{gr}}M}=M/M_1\oplus M/M_2 \oplus \cdots \oplus M/M_s,$$ such that each $M/M_k$ is gr-irreducible. Then $N\otimes V(i^n)$ is embedded into each $\Delta_{i^n}(M/M_k)$ and $\Delta_{i^n}(\hd M)$, which are quotients of $\Delta_{i^n}M$ by the exactness of $\Delta_{i^n}$. It follows from (i) that $\Delta_{i^n}(\hd M)\simeq \Delta_{i^n}(M/M_k) \simeq N\otimes V(i^n)$. Hence $\hd M$ must be gr-irreducible.
Moreover, we have $\varepsilon_i (\hd M)=\varepsilon_i (M)=n$.

(iii) Since we have proved $\Delta_{i^n}(\hd M)\simeq \Delta_{i^n}M$. Our assertion follows from the exactness of $\Delta_{i^n}$.
\end{proof}
\end{lemma}

\vskip 2mm
\begin{proposition}\label{P1}
{\it Let $i\in I$ and  $M\in R(\nu)$-$\fMod$ be a gr-irreducible module with $\varepsilon_i(M)=n$. Then $\Delta_{i^n}M$ is isomorphic to $K\otimes V(i^n)$ for some gr-irreducible $K\in R(\nu-ni)$-$\fMod$ with $\varepsilon_i(K)=0$. Furthermore, $M\simeq \hd\Ind_{\nu-ni,ni}K\otimes V(i^n)$ in this case.
}
\begin{proof}
Choose a gr-irreducible submodule $K\otimes V(i^n)$ of $\Delta_{i^n}M$, then we have $\varepsilon_i(K)=0$ by Lemma \ref{L1}. By Frobenius reciprocity (\ref{Frob}) and the irreducibility of $M$,
$M$ is a quotient of $\Ind_{\nu-ni,ni}K\otimes V(i^n)$. Therefore, $\Delta_{i^n}M$ is a quotient of  $\Delta_{i^n}(\Ind_{\nu-ni,ni}K\otimes V(i^n))$, which is isomorphic to $K\otimes V(i^n)$ according to Lemma \ref{L2} (i). Now, $\Delta_{i^n}M\simeq K\otimes V(i^n)$ since $K\otimes V(i^n)$ is gr-irreducible.

Since we have a surjective map $\Ind_{\nu-ni,ni}K\otimes V(i^n)\twoheadrightarrow M$ and since $\hd\Ind_{\nu-ni,ni}K\otimes V(i^n)$ is gr-irreducible by Lemma \ref{L2} (ii), we see that $M\simeq \hd\Ind_{\nu-ni,ni}K\otimes V(i^n)$.
\end{proof}
\end{proposition}

\vskip 2mm
\begin{corollary}\label{Coo}
{\it Let $i\in I$ and  $M,M' \in R(\nu)$-$\fMod$ are gr-irreducible module with $\varepsilon_i(M)=\varepsilon_i(M')=n$. Assume $M\not\simeq M'$ and
$$\Delta_{i^n}M\simeq K\otimes V(i^n), \ \ \Delta_{i^n}M'\simeq K'\otimes V(i^n)$$
 for gr-irreducible $K,K'\in R(\nu-ni)$-$\fMod$ with $\varepsilon_i(K)=\varepsilon_i(K')=0$. Then $K\not\simeq K'$.
}
\begin{proof}
If  $K\simeq K'$, then
$M\simeq \hd\Ind_{\nu-ni,ni}K\otimes V(i^n)\simeq \hd\Ind_{\nu-ni,ni}K'\otimes V(i^n)\simeq M'$ by Proposition \ref{P1}. This proves our claim.
\end{proof}
\end{corollary}

\vskip 2mm

\begin{theorem}
{\it The map $\Ch:G_0(R(\nu))\rightarrow \Z[q,q^{-1}]\Seq(\nu)$ is injective.}
\begin{proof}
We prove the characters of gr-irreducible $R(\nu)$-modules in $\B_{\nu}$ are linearly independent over $\Z[q,q^{-1}]$ by induction on $\text{ht}(\nu)$. The case of $\text{ht}(\nu)=0$ is trivial. Assume for $\text{ht}(\nu)<n$, our assertion is true. Now, suppose $\text{ht}(\nu)=n$ and we are given a non-trivial linear composition
\begin{equation}\label{Ch}
\sum_{M}c_M\Ch M=0
\end{equation}
for some $M\in \B_{\nu}$ and some $c_M\in\Z[q,q^{-1}]$. Choose an $i\in I$. We show by downward induction on $k=n,\dots,1$ that $c_M=0$ for all $M$ with $\varepsilon_i(M)=k$.

If $k=n$ and $M\in \B_{\nu}$ such that $\varepsilon_i(M)=n$, we must have $\nu=ni$ and $M=V(i^n)$, our assertion is trivial. Assume for $1\leq k<n$, we have $c_M=0$ for all $L$ with $\varepsilon_i (L)>k$. Taking out the terms with $i^k$-tail in the rest of (\ref{Ch}), we obtain
$$\sum_{M:\varepsilon_i(M)=k}c_M\Ch (\Delta_{i^k}M)=0.$$
If $\Delta_{i^k}M\simeq K\otimes V(i^k)$ for a gr-irreducible $K\in R(\nu-ki)$-$\fMod$ with $\varepsilon_i(K)=0$, then $$\Ch (\Delta_{i^k}M)=\gdim V(i^k)\cdot \Ch K\cdot i^k.$$
By the inductive hypothesis and the Corollary \ref{Coo}, we get $c_M=0$ for all $M$ with $\varepsilon_i(M)=k$. Since each gr-irreducible $R(\nu)$-modules $M$ has $\varepsilon_i(M)>0$ for at least one $i\in I$, the theorem has been proved.
\end{proof}
\end{theorem}

\vskip 2mm

For each $\nu\in \N[I]$, `$\Ch$' induces an injective map of $\Q(q)$-vector space $\Ch :\Q(q)\otimes_{\Z[q,q^{-1}]}G_0(R(\nu))\hookrightarrow \Q(q)\Seq(\nu)$, which is dual to
$$\Q(q)\Seq(\nu)\longrightarrow U^-_{\nu}\stackrel{\Gamma_{\Q(q)}}{\longrightarrow}K_0(R(\nu))_{\Q(q)}.$$
It follows that $\Gamma_{\Q(q)}$ is surjective. Combine with the injectivity of $\Gamma_{\Q(q)}$, we obtain the following categorification of $U^-$.

\vskip 2mm

\begin{proposition}{\it
$\Gamma_{\Q(q)}:U^-\rightarrow K_0(R)_{\Q(q)}$ is an isomorphism.
}\end{proposition}

\vskip 2mm

We next consider the surjectivity of $\Gamma:_{\mathcal A}U^-\rightarrow K_0(R)$. The following several results can be proved by the same manner in \cite[Chapter 5]{K2005}.

\vskip 2mm

\begin{lemma}\label{L3}
{\it Let $i\in I$ and  $M\in R(\nu)$-$\fMod$ be a gr-irreducible module. Then for any $0\leq n\leq \varepsilon_i(M)$, the graded socle $\soc(\Delta_{i^n}M)$ is a gr-irreducible $R(\nu-ni)\otimes R(ni)$-module of the form $L\otimes V(i^n)$ with $\varepsilon_i(L)=\varepsilon_i(M)-n$.
}
\begin{proof}
Let $\varepsilon_i(M)=a$ and $\Delta_{i^a}M\simeq K\otimes V(i^a)$ for some gr-irreducible $K\in R(\nu-ai)$-$\fMod$. For each constituent $L\otimes V(i^n)$ of $\soc(\Delta_{i^n}M)$ with $\varepsilon_i(L)=a-n$, we have $$\Res_{\nu-ai,(a-n)i,ni}^{\nu-ni,ni}L\otimes V(i^n)\hookrightarrow \Res_{\nu-ai,(a-n)i,ni}^{\nu-ni,ni} \Delta_{i^n}M.$$
On the other hand, by the transitivity of the $\Res$, we obtain
$$\Res_{\nu-ai,(a-n)i,ni}^{\nu-ni,ni} \Delta_{i^n}M\simeq\Res_{\nu-ai,(a-n)i,ni}^{\nu-ai,ai}\Res_{\nu-ai,ai}^{\nu}M \simeq K\otimes V(i^{a-n})\otimes V(i^{n}).$$
Hence $\soc(\Delta_{i^n}M)$ must equal $L\otimes V(i^n)$.
\end{proof}
\end{lemma}

\vskip 2mm
Define the functor $e_i=\Res_{\nu-i}^{\nu-i,i}\circ \Delta_i:R(\nu)\text{-}\fMod\rightarrow R(\nu-i)\text{-}\fMod$. Then for $M\in R(\nu)\text{-}\fMod$, $\varepsilon_i(M)=\text{max}\{n\geq 0\mid e_i^n M\neq 0\}$.

\vskip 2mm

\begin{lemma}
{\it Let $i\in I$ and  $M\in R(\nu)$-$\fMod$ be a gr-irreducible module with $\varepsilon_i(M)>0$. Then $\soc(e_iM)$ is a gr-irreducible $R(\nu-i)$-module with $\varepsilon_i(\soc (e_iM))=\varepsilon_i(M)-1$.
}
\begin{proof}
Let $L$ be a gr-irreducible submodule of $e_iM$. Since $e_iM=\bigoplus_{\jj\in\Seq(\nu-i)}1_\jj\otimes 1_iM$, we have $(1_{\nu-i}\otimes x_i^l)e_iM=0$ for $l\gg 0$. By Schur's Lemma and Proposition \ref{Cen}, $z=\sum_{\ii\in\Seq(\nu),1\leq k\leq m}x_{k,\ii}$ ($m=\text{ht}(\nu)$) acts on $M$ by a scalar. Similarly, $z'=\sum_{\ii'\in\Seq(\nu-i),1\leq k\leq m-1}x_{k,\ii'}$ acts on $L$ by scalar and so $z-z'$ acts on $L$ by a scalar $c$. Since $L\subseteq 1_{\nu-i}\otimes 1_iM$, for every $m\in L$, we get
$$(z-z')m=(\sum_{\jj\in\Seq(\nu-i)}1_\jj\otimes x_i)m=1_{\nu-i}\otimes x_im=cm.$$
Now $(1_{\nu-i}\otimes x_i^l)m=0$ for $l\gg 0$ yields $c=0$, and so $(1_{\nu-i}\otimes x_i)L=0$. Hence $L$ is a gr-irreducible $R(\nu-i)\otimes R(i)$-submodule of $\Delta_iM$, which is isomorphic to $L\otimes V(i)$. By Lemma \ref{L3}, $\soc(\Delta_iM)$ is gr-irreducible. It follows that $\soc(e_iM)=L$ is gr-irreducible.
\end{proof}
\end{lemma}

\vskip 2mm

Let $i\in I$. For a gr-irreducible $M\in R(\nu)$-$\fMod$, define $\widetilde{e}_iM=\soc (e_iM)$. If $\varepsilon_i(M)>0$, $\widetilde{e}_iM$ is gr-irreducible with $\varepsilon_i(\widetilde{e}_iM)=\varepsilon_i(M)-1$.

\vskip 2mm

\begin{proposition}\label{P2}
{\it For a gr-irreducible $M\in R(\nu)$-$\fMod$ and $n\geq 0$, we have $$\soc(\Delta_{i^n}M)\simeq \widetilde{e}_{i}^nM\otimes V(i^n)\{r\}.$$
for some $r\in \Z$.}
\begin{proof}
The case of $i\in I^+$ has been proved in \cite[Lemma 3.13]{KL2009}. Assume $i\in I^-$, since $\widetilde{e}_{i}M\otimes V(i)$ is a graded submodule of $\Delta_iM$, we see that $\widetilde{e}_{i}^nM\otimes V(i)^{\otimes n}$ is a graded submodule of $\Res_{\nu-ni,i,\dots,i}^{\nu-ni,ni}\Delta_{i^n}M$. By the following Frobenius reciprocity
$$\HOM(\Ind_{\nu-ni,i,\dots,i}^{\nu-ni,ni}\widetilde{e}_{i}^nM\otimes V(i)^{\otimes n}, \Delta_{i^n}M)\simeq \HOM(\widetilde{e}_{i}^nM\otimes V(i)^{\otimes n}, \Res_{\nu-ni,i,\dots,i}^{\nu-ni,ni}\Delta_{i^n}M),$$
we have a nonzero homomorphism from $\widetilde{e}_{i}^nM\otimes \Ind_{i,\dots,i}^{ni}V(i)^{\otimes n}$ to $\Delta_{i^n}M$. The composition factors of $\widetilde{e}_{i}^nM\otimes \Ind_{i,\dots,i}^{ni}V(i)^{\otimes n}$ can only be $\widetilde{e}_{i}^nM\otimes V(i^n)$, up to degree shifts. So we obtain $\widetilde{e}_{i}^nM\otimes V(i^n)\{r\}\hookrightarrow \Delta_{i^n}M$ for some $r\in \Z$. Now our assertion follows from Lemma \ref{L3}.
\end{proof}
\end{proposition}

\vskip 2mm
\begin{lemma}\label{L4}{\it
Let $i\in I$ and $M\in R(\nu)$-$\fMod$ be a gr-irreducible module with $\varepsilon_i(M)=n$. We have  $M\simeq \hd\Ind_{\nu-ni,ni}\widetilde{e}_{i}^nM\otimes V(i^n)$,
up to a degree shift.
}\end{lemma}
\begin{proof}
The lemma follows from Proposition \ref{P1} and \ref{P2}.
\end{proof}
\vskip 2mm

Assume $|I|=k$. The elements in $I$ is labelled by $i_0,\dots, i_p,i_{p+1},\dots, i_{k-1}$,
such that $i_0,\dots , i_p\in I^-$ and $i_{p+1},\dots , i_{k-1}\in I^+$. For $r\geq k$, define $i_r=i_{r'}$ where $r'$ is the residue of $r$ modulo $k$. For $b\in \B_{\nu}$, assign the sequence $W_{b}=c_0c_1\dots$ of nonnegative integers:
$c_0=\varepsilon_{i_0}(S_b)$, and let $M_1=\widetilde{e}_{i_{0}}^{c_0}(S_b)$. Inductively, $c_r=\varepsilon_{i_r}(M_r)$ and $M_{r+1}=\widetilde{e}_{i_r}^{c_r}(M_r)$. We have $c_0+c_1+\cdots=\text{ht}(\nu)$ and only finitely many terms in the sequence are nonzero. Note that if $b\neq b'$, then $W_b\neq W_{b'}$ by Lemma \ref{L4}.

Introduce a lexicographic order on sequences of nonnegative integers: $c_0c_1\dots>d_0d_1\dots$ if for some $t$, $c_0=d_0, c_1=d_1,\dots,c_{t-1}=d_{t-1}$ and $c_t>d_t$. We set $b>b'$ in $\B_{\nu}$ if and only if $W_b>W_{b'}$. To each $b\in \B_{\nu}$, assume $W_{b}=c_0c_1\dots$, assign the projective $R(\nu)$-module $P_{{W_b}^{\bullet}}$ associated to the  sequence
$W_b^{\bullet}=\cdots i_k^{c_k}i_{k-1}^{(c_{k-1})}\cdots i_{p+1}^{(c_{p+1})}i_p^{c_p}\cdots i_0^{c_0}$.

\vskip 2mm

\begin{proposition}
{\it $\HOM(P_{{W_b}^{\bullet}}, S_{b'})=0$ if $b>b'$ and $\HOM(P_{{W_b}^{\bullet}}, S_b)\simeq \K$.
}
\begin{proof}
For $i\in I^+$,  we have $\HOM(P_{i^{(n)}},V(i^n))\simeq \K$ since $P_{i^{(n)}}$ is the graded projective cover of $V(i^n)$. For $i\in I^-$, $\HOM(R_{ni},V(i^n))\simeq V(i^n)\simeq \K$ as graded vector spaces. The results follow immediately from the Frobenius reciprocity and Proposition \ref{P2}.
\end{proof}
\end{proposition}

\vskip 2mm

By proposition above, each $[P]\in K_0(R(\nu))$ can be written as a $\Z[q,q^{-1}]$-linear combination of $[P_{{W_b}^{\bullet}}]$ for $b\in \B_{\nu}$. Therefore, $\Gamma$ is surjective. We obtain
\vskip 2mm
\begin{theorem}{\it
 $\Gamma:_{\mathcal A}U^-\rightarrow K_0(R)$ is an isomorphism.
}\end{theorem}

\vskip 2mm
For $M\in R(\nu)$-$\fMod$, let $M^\divideontimes=\HOM_\K(M,\K)^\psi$ be the dual module in $R(\nu)$-$\fMod$ with the action given by
$$(zf)(m):=f(\psi(z)m)\ \text{for}\ z\in R(\nu), f\in \HOM_\K(M,\K), m\in M.$$
 As proved in \cite[Section 3.2]{KL2009}, for each gr-irreducible $R(\nu)$-module $S$, there is a unique $r\in\Z$ such that $(L\{ r \})^\divideontimes \simeq L\{r\}$, and the graded projective cover of $L\{ r \}$ is stable under the bar-involution $^-$.

\vskip 6mm

\bibliographystyle{amsplain}

\begin{thebibliography}{99}
\bibitem{Kang}
S.-J. Kang, \emph{Quantum deformations of generalized Kac-Moody algebras and their
modules}, J. Algebra, \textbf{175}(3) (1995) 1041-1066.
\bibitem{KOP2012}
 S.-J. Kang, S.-j. Oh, E. Park, \emph{Categorification of Quantum Generalized Kac-Moody Algebras and Crystal Bases}, Internat. J. Math., \textbf{23}(11) (2012) 13-88.
\bibitem{K2005}
A. Kleshchev, \emph{Linear and projective representations of symmetric groups}, volume 163 of
Cambridge Tracts in Mathematics. Cambridge U. Press, 2005.
\bibitem{KL2009}
M. Khovanov and A. Lauda, \emph{A diagrammatic approach to categorification of quantum groups $\rm{I}$}, Represent. Theory, \textbf{13} (2009) 309-347.
\bibitem{KL2011}
M. Khovanov and A. Lauda, \emph{A diagrammatic approach to categorification of quantum groups $\rm{II}$}, Trans. Amer. Math. Soc., \textbf{363}(5) (2011) 2685-2700.
\bibitem{Rou}
R. Rouquier, \emph{$2$-Kac-Moody algebras}, preprint, arXiv:0812.5023.
\bibitem{Rou1}
 R. Rouquier, \emph{Quiver Hecke algebras and $2$-Lie algebras}, Algebra Colloq., \textbf{19} (2012)
359-410
\bibitem{SV2001}
B. Sevenhant and M. Van den Bergh, \emph{A relation between a conjecture of Kac and the structure of the Hall algebra}, J. Pure Appl. Algebra, \textbf{160} (2001), 319-332.
\bibitem{VV2011}
 M. Varagnolo and E. Vasserot, \emph{Canonical bases and KLR-algebras}, J. Reine Angew.
Math., \textbf{659} (2011) 67-100.
\end{thebibliography}

%%%%%%%%
\end{document}